\documentclass[12pt]{amsart}%
\usepackage{amsfonts}
\usepackage{graphicx}
\usepackage{rotating}
\usepackage{pstricks, pst-node, pst-text, pst-3d}
\usepackage{amsmath}
\usepackage{amssymb}%
\usepackage{caption}
\input{amssym.def}
\newtheorem{theorem}{Theorem}

\newtheorem{corollary}{Corollary}

\newtheorem{remark}{Remark}

\theoremstyle{remark}

\newcommand{\re}{\text{\rm Re}\,}
\newcommand{\im}{\text{\rm Im}\,}
\newcommand{\Hol}{\text{\rm Hol}\,}
\newcommand{\mob}{\text{\it m\"ob}\,}

\input epsf
\begin{document}
\title[Boundary distortion estimate]{Boundary distortion estimates for holomorphic maps}
\author[A.~Frolova, M.~Levenshtein, D.~Shoikhet, A.~Vasil'ev]{Anastasia Frolova$^{\dag}$, Marina Levenshtein,   David Shoikhet, and Alexander Vasil'ev$^{\dag}$}

\address{A.~Frolova: Department of Mathematics, University of Bergen, Johannes Brunsgate 12, Bergen 5008, Norway}
\email{Anastasia.Frolova@math.uib.no}

\address{M.~Levenshtein: Department of Mathematics, ORT Braude College,
P.O. Box 78, 21982 Karmiel, Israel}
\email{marlev@braude.ac.il}

\address{D.~Shoikhet: Department of Mathematics, ORT Braude College,
P.O. Box 78, 21982 Karmiel, Israel}
\email{davs@braude.ac.il}

\address{A.~Vasil'ev: Department of Mathematics, University of Bergen, Johannes Brunsgate 12, Bergen 5008, Norway}
\email{Alexander.Vasiliev@math.uib.no}

\thanks{All authors are supported by EU FP7 IRSES program STREVCOMS, grant  no. PIRSES-GA-2013-612669; The authors$^{\dag}$ have been  supported by the grants of the Norwegian Research Council \#204726/V30, \#213440/BG}

\subjclass[2010]{Primary 30C35, 30D05; Secondary 37C25, 30C75}

\keywords{Fixed point, semigroup of analytic functions,
Denjoy-Wolff point, reduced modulus, digon, angular derivative}

\begin{abstract}
We establish some estimates  of the  the angular derivatives from below for holomorphic self-maps of the unit disk $\mathbb D$ at one and two fixed points of the unit circle
provided there is no fixed point inside $\mathbb D$. The results complement  Cowen-Pommerenke  and Anderson-Vasil'ev type estimates in the case of univalent functions.
We use the method of extremal length and propose a new semigroup approach to deriving inequalities for holomorphic self-maps of the disk which are not necessarily univalent using known inequalities for univalent functions. This approach allowed us to receive a new Ossermans type estimate as well as inequalities for holomorphic self-maps which images do not separate the origin and the boundary.
\end{abstract}
\maketitle

\section{Introduction}

Let $\mathbb{D}=\{z\in\mathbb{C}:\,\,|z|<1\}$ be the unit disk, and let
$\mathrm{Hol}(\mathbb{D},\mathbb{D})$ stand for the family of analytic self-maps
of \thinspace$\mathbb{D}$.
The family $\mathrm{Hol}(\mathbb{D},\mathbb{D})$ forms a semigroup with
respect to the functional composition with the identity map as the unity. The study of fixed points of elements from  $\mathrm{Hol}(\mathbb{D},\mathbb{D})$   always
plays a prominent role in the theory of dynamical systems. We recall that a point $\xi\in\hat{\mathbb{D}}$ from the closure $\hat{\mathbb{D}}$ of $\mathbb D$
is said to be a {\it fixed point} of an element $\varphi\in\mathrm{Hol}
(\mathbb{D},\mathbb{D})$ if $\lim_{r\rightarrow1^-}\varphi(r\xi)=\xi.$
In particular, if $\xi\in\mathbb T:=\partial\mathbb{D}$, then the above definition is equivalent to the assertion
that the angular limit $\angle\lim\limits_{z\rightarrow\xi}\varphi(z)=\xi$ exists, i.~e.,
 $\lim\limits_{z\rightarrow\xi,\,\,z\in\Delta_{\xi}}%
\varphi(z)=\xi$ for any Stolz angle $\Delta_{\xi}$ centered at $\xi$, see,~e.~g., \cite[Corollary 2.17, page 35]{Pomm1}. Such
points $\xi\in\mathbb T$ are usually called {\it boundary fixed points} of
$\varphi$.  Recall that the angular limit $\varphi(\xi)$ exists for almost all $\xi \in \mathbb T$, moreover, the exceptional set in $\mathbb T$ is of capacity zero. The classification of the fixed points of $\varphi\in\mathrm{Hol}(\mathbb{D},\mathbb{D})$ can be performed  regarding the value of the derivative $|\varphi'(\xi)|$ in the case $\xi\in \mathbb D$, or
the value of the {\it angular derivative}
\[
\varphi^{\prime}(\xi):=\angle\lim\limits_{z\rightarrow\xi}\frac{\varphi
(z)-\xi}{z-\xi}.
\]
in the case $\xi\in \mathbb T$, which is real and $\varphi^{\prime}(\xi)\in(0,+\infty)\cup\{\infty\}$ in this case, that follows from the {\it Julia-Wolff lemma}, see,~e.~g., \cite[Proposition 4.13, page 82]{Pomm1}. We recall that the angular derivative at a boundary point $\xi$ exists if and only if the analytic function $\varphi'(z)$ has the angular limit $\angle\lim\limits_{z\to\xi}\varphi'(z)$, see,~e.~g., \cite[Proposition 4.7, page 79]{Pomm1}. Whenever
$\varphi^{\prime}(\xi)\neq\infty\,,$ for a boundary fixed point $\xi$, we say that $\xi$ is a regular (boundary)
fixed point. The regular fixed points can be {\it attractive} if
$\varphi^{\prime}(\xi)\in(0,1),$ {\it neutral} if $\varphi^{\prime}(\xi)=1$ or
{\it repulsive} if $\varphi^{\prime}(\xi)\in(1,+\infty).$ Fixed points $\xi
\in\mathbb{D}$ with $\,\,|\varphi^{\prime}(\xi)|<1$ are also called attractive.

The existence of the angular derivative is a difficult problem in general, however the Julia-Wolff theory implies that in the case of $\varphi\in\mathrm{Hol}(\mathbb{D},\mathbb{D})$, the angular
derivative $\varphi'(\xi)$ exists (but perhaps infinite) at all points $\xi\in \mathbb T$, where the angular limit $\varphi(\xi)$ exists and $|\varphi(\xi)|=1$.
Furthermore, the mapping at the point $\xi$ may be conformal ($0<|\varphi'(\xi)|<\infty$) or twisting. The McMillan twist theorem,~see \cite[page 127]{Pomm1}, states that $\varphi$ is conformal for almost all such points.
A classical result by Denjoy and Wolff \cite{Denjoy, Wolff}, states that for a holomorphic self-map
$\varphi$ of the unit disk $\mathbb{D}$ different from a (hyperbolic) rotation,
there exists a unique fixed point $\tau\in\hat{\mathbb{D}}$ such that the
sequence of the iterates $(\varphi_{n}(z))$,  defined by
$\varphi_0(z) = z, \varphi_n(z) = \varphi (\varphi_{n-1}(z)), n = 1, 2, \dots$, converges locally uniformly on
$\mathbb{D}$ to $\tau$ as $n\rightarrow\infty$. This point $\tau$ is called the
{\it Denjoy-Wolff point} of $\varphi$ and it is the only fixed
point of $\varphi$ satisfying $\varphi^{\prime}(\tau)\in\mathbb{D}$. The  point $\tau$ is the only attractive fixed point of $\varphi$ in the above
multiplier sense.

The Julia-Carath\'eodory Theorem \cite{Cara} and the Wolff Lemma \cite{Wolff2}  imply that if $\varphi$ has no interior fixed point, then there exists a boundary fixed point $\xi$ such that the angular derivative $\varphi'(\xi)$ exists and $\varphi'(\xi)\in(0,1]$. The mapping  $\varphi$ is said to be of {\it parabolic type} if $\varphi'(\xi)=1$, and of {\it hyperbolic type} if $\varphi'(\xi)\in(0,1)$. Otherwise, the mapping $\varphi$ has an interior fixed
point $\tau \in \mathbb D$, and for each boundary fixed point $\xi\in \mathbb{T}$, $\varphi'(\xi)>1$.
Quantification of the latter statement for the case $\tau=0$ was first given by
Unkelbash \cite{Unkel},  and rediscovered by  Osserman in \cite{Osserman} 60 years later. They proved that if $\varphi$ has a regular boundary fixed point at $1$, and $\varphi(0)=0$, then
\[
\varphi'(1)\geq \frac{2}{1+|\varphi'(0)|}.
\]
In the study of the case of several fixed boundary points, a real breakthrough was made by Cowen and Pommerenke \cite{CowenPommerenke}. Summarising  their results
and adding recent progress by Elin, Shoikhet, Tarkhanov, and  Bolotnikov \cite{Bolotnikov, Elin} we formulate the following theorem

\medskip
\noindent
{\bf Theorem A.}
{\it Let $\varphi \in \mathrm{Hol}(\mathbb{D},\mathbb{D})$, and let $\tau$ be the Denjoy-Wolff point of $\varphi$ and $\xi_1,\dots,\xi_n$ be other possible distinct fixed points
of $\varphi$ in ${\mathbb T}$.
\begin{itemize}
\item If $\tau=0$, then
\begin{equation}\label{eq1}
\sum\limits_{j=1}^n\frac{1}{\varphi'(\xi_j)-1}\leq\re\frac{1+\varphi'(0)}{1- \varphi'(0)};
\end{equation}
\item If $\tau=1$  and $\varphi'(1)\in (0,1)$ (hyperbolic attractor), then
\begin{equation}\label{eq2}
\sum\limits_{j=1}^n\frac{1}{\varphi'(\xi_j)-1}\leq\frac{\varphi'(1)}{1- \varphi'(1)};
\end{equation}
\item If $\tau=1$ (including $\varphi'(1)=1$, parabolic attractor), then
\begin{equation}\label{eq3}
\sum\limits_{j=1}^n\frac{|1-\xi_j|^2}{\varphi'(\xi_j)-1}\leq 2\re \left(\frac{1}{\varphi(0)}-1\right);
\end{equation}
\end{itemize}
All estimates are sharp and the extremal functions satisfy some functional equations provided in \cite{Bolotnikov, CowenPommerenke, Elin}.
}
\medskip

In the case of univalent functions, the following theorem holds \cite{CowenPommerenke}.

\medskip
\noindent
{\bf Theorem B.}
{\it Let $\varphi\in\mathrm{Hol}(\mathbb{D},\mathbb{D})$ be
univalent with an attractive Denjoy-Wolff point $\tau\in\mathbb{T}$, and let
$\xi_{1},\dots,\xi_{n}$ be $n\,\ $different repulsive boundary fixed points of
$\varphi.$ Then,%
\[
\sum\limits_{k=1}^{n}\frac{1}{\log\varphi^{\prime}(\xi_{k})}\leq-\frac{1}%
{\log\varphi^{\prime}(\tau)}.
\]
Moreover, this inequality is sharp.}

\medskip
A weighted version of this theorem was proved by Contreras, D{\'\i}az-Madrigal and Vasil'ev in \cite{CDV}.

For two boundary fixed points $\xi_1$ and $\xi_2$, without loss of generality,  applying rotation, we assume that  $\xi_1=e^{-i\theta}$,  $\xi_2=e^{i\theta}$, and $\theta\in (0,\frac{\pi}{2}]$. The following result is a consequence  of \cite[Theorem 3.1]{CowenPommerenke}.

\medskip
\noindent
{\bf Theorem C.}
{\it Let $\varphi\in\mathrm{Hol}(\mathbb{D},\mathbb{D})$ and let $\xi_1=e^{-i\theta}$ and $\xi_2=e^{i\theta}$ be fixed.  Then,
\begin{equation}\label{CP31}
\varphi'(e^{i\theta})\varphi'(e^{-i\theta})\geq \sup\limits_{z\in\gamma}\left(1+\frac{4\im \varphi(z)}{(1-|\varphi(z)|^2)^2}\right),
\end{equation}
where $\gamma$ is the hyperbolic geodesic joining $\xi_1$ and $\xi_2$.
}

\medskip

In this paper we are aimed at a sharp analogue of inequality \eqref{eq3}  in Theorem~A and  \eqref{CP31}  in Theorem~C  for univalent self-maps of the unit disk in the case of two fixed boundary points. The result is presented in Section \ref{sec:two}. 
The method of the proof is based on the notion of digon and extremal partitions of a domain, which was already successfully applied for this type of problems with
angular derivatives, see  \cite{AV, CDV, PomVas, VasBook}.

In this paper we also present an approach which allows one to obtain estimates for  functions from the general class $\Hol(\mathbb D,\mathbb D)$ by means of estimates for univalent functions. The approach is based on the theory of semigroups of analytic functions.

In Section \ref{sec:one} we combine both methods presented in this paper. First we reprove a known sharp estimate
\begin{equation}
\label{CP}
\varphi'(1)\geq  \re\frac{1-\varphi(0)}{1+\varphi(0)}
\end{equation}
for univalent functions with the fixed boundary point at $1$ using moduli and extremal partitions. By doing this we demonstrate briefly the essence of the method we use to prove the result for function with two boundary fixed points. Then we use a semigroup approach to extend the inequality (\ref{CP}) to  functions from $\mathrm{Hol}(\mathbb{D},\mathbb{D})$.

In general, using this technique one can start with an estimate for univalent functions and arrive at different type of inequality for general analytic function.
For example, one can obtain an Osserman's estimate for functions in $\mathrm{Hol}(\mathbb{D},\mathbb{D})$ from Anderson-Vasil'ev inequality for univalent functions. This and other examples of transformation
of inequalities is presented in Section \ref{sec:one}.

\section{Preliminaries}
\label{sec:pre}
The definitions given here are specified for one digon in a domain in $\hat{\mathbb C}$. For general formulations in the case of admissible families of digons
and extremal partitions of Riemann surfaces for weighted sums of the moduli, see \cite{Em3, Kuz4, Sol7, VasBook}.

\subsection{Reduced modulus of digon}

Let $D$ be a hyperbolic simply connected domain in $\mathbb C$
with two finite fixed boundary points  $a$, $b$ (maybe with the
same support) on its piecewise smooth boundary. It is called a
digon. Denote by $S(a,\varepsilon)$ a region that is the
connected component of $D\cap\{|z-a|<\varepsilon\}$ with the point
$a$ in its border. Denote by $D_{\varepsilon}$ the domain
$D\setminus \{S(a,\varepsilon_1)\cup S(b,\varepsilon_2)\}$ for
sufficiently small $\varepsilon_{1,2}$ such that there is a curve in $D_{\varepsilon}$
connecting the opposite sides on $S(a,\varepsilon_1)$ and
$S(b,\varepsilon_2)$. Let $M(D_{\varepsilon})$ be the modulus of
the family of paths in $D_{\varepsilon}$ that connect the boundary
arcs of $S(a,\varepsilon_1)$ and $S(b,\varepsilon_2)$ when lie in
the circumferences $|z-a|=\varepsilon_1$ and $|z-b|=\varepsilon_2$
(we choose a single arc in each circle so that both arcs can be
connected in $D_{\varepsilon}$). If the limit
\begin{equation}
m(D,a,b)=\lim\limits_{\varepsilon_{1,2}\to 0}\left(
\frac{1}{M(D_{\varepsilon})}+
\frac{1}{\varphi_a}\log\,\varepsilon_1+\frac{1}{\varphi_b}\log\,\varepsilon_2\right),\label{eq:moddig}
\end{equation}
exists, where $\varphi_a=\sup\,\Delta_a$ and
$\varphi_b=\sup\Delta_b$ are the inner angles, where $\Delta_{a}$ and $\Delta_{b}$
are the Stolz angles inscribed in $D$ at $a$ and $b$ respectively,
then $m(D,a,b)$ is called the  reduced modulus of the digon $D$.
Various conditions guarantee the existence of this modulus,
whereas even in the case of a piecewise analytic boundary there are
examples \cite{Sol7} which show that this is not always the case.
The existence of  limit (\ref{eq:moddig}) is a local
characteristic of the domain $D$ (see \cite{Sol7}, Theorem 1.2).
If the domain $D$ is conformal (see the definition in \cite[page 80]{Pomm1}) at the points $a$ and $b$, then (\cite{Sol7}, Theorem
1.3)   limit (\ref{eq:moddig}) exists. More generally, suppose
that there is a conformal map $f(z)$ of the domain
$S(a,\varepsilon_1)\subset D$ onto a circular sector, so that the
angular limit $f(a)$ exists which is thought of as a vertex of
this sector of angle $\varphi_a$. If the function $f$ has the
finite non-zero angular derivative $f'(a)$ we say that the domain
$D$ is also conformal at the point $a$ (compare \cite[page 80]{Pomm1}). If the digon $D$ is conformal at the points $a$ and $b$
then  the limit (\ref{eq:moddig}) exists (\cite{Sol7}, Theorem
1.3). It is noteworthy that Jenkins and Oikawa \cite{JenOik}
in 1977 applied extremal length techniques to study the behaviour
of a regular univalent map at a boundary point. Necessary and
sufficient conditions were given for the existence of a finite
non-zero angular derivative. Independently a similar result was
 obtained by Rodin and Warschawski \cite{Rodin}.

The reduced modulus of a digon is not invariant under conformal mappings.
The following result gives a change-of-variable formula, see, e.g., \cite{VasBook}.  Let a digon $D$ with the vertices at $a$ and $b$ be so that 
limit (\ref{eq:moddig}) exists and the Stolz angles are
$\varphi_a$ and $\varphi_b$ . Suppose that there is a conformal
map $f(z)$ of the digon $D$ (which is conformal at $a$ and $b$) onto
a digon $D'$, so that there exist the angular limits $f(a)$ and
$f(b)$ with the inner angles $\psi_a$ and $\psi_b$ at the vertices
$f(a)$ and $f(b)$ which we also understand as the supremum over
all Stolz angles inscribed in $D'$ with the vertices at $f(a)$ or
$f(b)$, respectively. If the function $f$ has the  finite non-zero
angular derivatives $f'(a)$ and $f'(b)$, then
$\varphi_a=\psi_{a}$, $\varphi_b=\psi_{b}$, and the reduced
modulus (\ref{eq:moddig}) of $D'$ exists and changes
\cite{Em3, Kuz4, Sol7, VasBook} according to the rule
\begin{equation}
m(f(D),f(a),f(b))= m(D,a,b)+ \frac{1}{\psi_a}\log
|f'(a)|+\frac{1}{\psi_b}\log |f'(b)|.\label{eq:modchange1}
\end{equation}

If we suppose, moreover, that $f$ has the expansion
$$
f(z)=w_1+(z-a)^{\psi_a/\varphi_a}(c_1+c_2(z-a)+\dots)
$$
in a neighborhood of the point $a$, and the expansion
$$
f(z)=w_2+(z-b)^{\psi_b/\varphi_b}(d_1+d_2(z-b)+\dots)
$$
in a neighborhood of the point $b$, then
the reduced modulus of  $D$
changes according to the rule
\begin{equation}
m(f(D),f(a),f(b))= m(D,a,b)+
\frac{1}{\psi_a}\log |c_1|+\frac{1}{\psi_b}\log |d_1|.\label{eq:modchange2}
\end{equation}
Obviously, one can extend this definition to the case of vertices
with infinite support.

\subsection{Extremal partition by digons}

Let $\Omega$ be a hyperbolic domain in $\hat{\mathbb C}$ that has a finite number of hyperbolic and parabolic boundary components.
We consider a family $\mathcal F$ of digons $D$ in $\Omega$ with two fixed vertices $a$ and $b$ on $\partial \Omega$, such that any arc connecting
the vertices of $D\in \mathcal F$ is not homotopic to a point of $\Omega$. The boundary points are understood in the Carath\'eodory sense.

We require the digons from $\mathcal F$ to be
conformal at their vertices.
A general theorem, see \cite{Em3, Kuz4, Sol7, VasBook}, implies that
any
collection of  admissible digons $\mathcal F$  satisfies the
inequality
$m(D,a,b)\geq m(D^*,a,b)$,
with the equality sign only for $D=D^*$. Here
 $D\sp*$ is  a strip domain in the trajectory structure
of a unique quadratic differential $Q(\zeta) d\zeta\sp2$, and
there is a conformal map $g(\zeta)$,
$\zeta\in D\sp*$ that satisfies the differential equation
\begin{equation}\label{eq_g}
 \left( \frac{g'(\zeta)}{g(\zeta)} \right)\sp2=
4\pi\sp2 Q(\zeta),
\end{equation}
 and which maps
$D\sp*$ onto the strip $\mathbb C\setminus [0,\infty)$.
The critical trajectories of $Q (\zeta) d\zeta\sp2$ define in
$\Omega$ a  strip domain $D^*$ associated  with $\mathcal F$.

\subsection{Semigroups of analytic functions}

Let us recall that a \textit{(one-parameter) semigroup of analytic functions}
is any continuous homomorphism $\Phi:t\mapsto\Phi(t)=\varphi_{t}$ from the
additive semigroup of non-negative real numbers into the composition semigroup
of all analytic functions which map $\mathbb{D}$ into $\mathbb{D}$. That is,
$\Phi$ satisfies the following three conditions:

\begin{enumerate}
\item[a)] $\varphi_{0}$ is the identity in $\mathbb{D},$

\item[b)] $\varphi_{t+s}=\varphi_{t}\circ\varphi_{s},$ for all $t,s\geq0,$

\item[c)] $\varphi_{t}(z)$ tends to $z$ locally uniformly in $\mathbb{D}$ as
$t\rightarrow0$.
\end{enumerate}

It is well-known that the functions $\varphi_{t}$ are always univalent. If $a$
is the Denjoy-Wolff point of one of the functions $\varphi_{t_{0}}$, for some
$t_{0}>0,$ then $a$ is the Denjoy-Wolff point of all the functions of the
semigroups, that is, all  functions of a semigroup share the Denjoy-Wolff
point. Moreover, if a point $\xi\in\partial\mathbb{D}$ is a boundary fixed
point of $\varphi_{t_{0}}$ for some $t_{0}>0,$ then it is a boundary fixed
point of all $\varphi_{t}$ \cite{ContrDiazPom}.

Given a semigroup $(\varphi_t)_{t\geq 0}$, it is well-known (see \cite{BerkPorta, Shoikhet}) that there exists a unique analytic function $g\colon \mathbb D\to\mathbb C$ such that
\[
\frac{d\varphi_t}{dt}=g(\varphi_t),
\]
for all $z\in \mathbb D$ and $t\geq 0$, called the {\it infinitesimal generator} of the semigroup  $(\varphi_t)_{t\geq 0}$. The Berkson-Porta representation \cite{BerkPorta} assures
that an analytic function $g\colon \mathbb D\to\mathbb C$  is the infinitesimal generator of a semigroup of analytic functions $(\varphi_t)_{t\geq 0}$ if and only if there exists
a point $w\in\hat{\mathbb D}$ and an analytic function $p\colon \mathbb D\to\mathbb C$ with $\re p(z)>0$ in $\mathbb D$, such that
\[
g(z)=(w-z)(1-\bar{w}z)p(z),\quad z\in\mathbb D.
\]
Such a representation is unique. If $(\varphi_t)$ is not the trivial group of the identity maps, then $w$ is either the DenjoyÐWolff point of the semigroup in the case where  $(\varphi_t)_{t\geq 0}$ is not a group of hyperbolic rotations, or the unique interior fixed point, otherwise.

\section{One fixed boundary point}
\label{sec:one}
\subsection{Cowen-Pommerenke type inequality}
In this section we will prove a known theorem by Cowen and Pommerenke
showing, in particular, a method how one can obtain estimates for general functions from $\Hol(\mathbb D,\mathbb D)$ by means of estimates for univalent functions.  At the same time 
we will show an application of the reduced moduli of digons and  extremal partitions
infinitesimal generators for semigroups.

\begin{theorem}\cite[Theorem 8.1]{CowenPommerenke}
If $\varphi \in \mathrm{Hol}(\mathbb{D},\mathbb{D})$, and $\varphi(1)=1$, then the sharp estimate
\begin{equation}\label{CPe}
\varphi'(1)\geq \frac{1}{\re\frac{1+\varphi(0)}{1-\varphi(0)}},
\end{equation}
holds with the equality sign for the M\"obius map
\[
m(z)=\frac{1+\bar{a}}{1+a}\frac{z+a}{1+z\bar{a}}, \quad a\in\mathbb D,\quad a\frac{1+\bar{a}}{1+a}=\varphi(0).
\]
\end{theorem}

\begin{corollary}
If $\varphi \in \mathrm{Hol}(\mathbb{D},\mathbb{D})$, and $\varphi(1)=1$, then the sharp estimate
\begin{equation}\label{ineqGuniv}
\varphi'(1)\geq \re\frac{1-\varphi(0)}{1+\varphi(0)},
\end{equation}
holds with the equality sign only for real values of $\varphi(0)$ and for the M\"obius map
\[
m(z)=\frac{z+\varphi(0)}{1+z{\varphi}(0)}.
\]
\end{corollary}

\begin{proof} We start with the univalent case.
Let us consider the family $\mathcal F_0$ of digons $D_0$ in $\Omega=\mathbb D\setminus \{0\}$ with two vertices $1^+$ and $1^-$  over the same point $1$  and the equal angles $\pi/2$ at these vertices,  such that
any arc connecting $1^+$ and $1^-$ in $D_0$ starting at one of the vertices comes to the other in $\Omega$ making the round about the origin. Then
\begin{equation}\label{P1}
\min\limits_{D_0\in \mathcal F} m(D_0,1^+,1^-)= m(D_0^*,1^+,1^-)=0,
\end{equation}
where $D^*_0=\mathbb D\setminus [0,1)$, which is a strip domain in the trajectory structure of the quadratic differential
\[
Q_0(z)dz^2=\frac{1}{(z-1)^2z}dz^2,\quad z\in \Omega.
\]

Let $\varphi \in \mathrm{Hol}(\mathbb{D},\mathbb{D})$ be an arbitrary univalent map, $\varphi(\xi)=\xi$, $\xi\in\partial\mathbb{D}$. Denote by $D_{\xi}^*$ the unit disk with a slit along the closed line segment connecting the origin and the point $\xi$. We regard $D_{\xi}^*$ as a digon with vertices at $\xi$. It can be obtained from the domain $D^*_0$ by rotation. The rotation transform does not change the reduced modulus of a digon. Therefore, the reduced modulus $m(D_{\xi}^*,\xi^+,\xi^-)=0$.

We observe that $\sigma\circ\varphi(D_{\xi}^*)$ is an admissible domain in the problem
of extremal partition \eqref{P1}, where

\[
\sigma(z)=\frac{1-\bar{\varphi}(0)\xi}{\xi-\varphi(0)}\frac{z-\varphi(0)}{1-\bar{\varphi}(0)z},\quad |\sigma'(\xi)|=\frac{1-|\varphi(0)|^2}{|1-{\varphi}(0)\bar{\xi}|^2}.
\]
This and  relation \eqref{eq:modchange1} imply that
\[
\frac{4}{\pi}\log |\varphi'(\xi)|+\frac{4}{\pi}\log \frac{1-|\varphi(0)|^2}{|1-{\varphi}(0)\bar{\xi}|^2}\geq 0,
\]
or
\begin{equation}
\label{inegGuniv2}
|\varphi'(\xi)|\geq \frac{|1-{\varphi}(0)\bar{\xi}|^2}{1-|\varphi(0)|^2}\geq \re\frac{1-\varphi(0)\bar{\xi}}{1+\varphi(0)\bar{\xi}}=\re\frac{\xi-\varphi(0)}{\xi+\varphi(0)}.
\end{equation}

When $\xi=1$, the first inequality is equivalent to \eqref{CPe} and the last one is equivalent to (\ref{ineqGuniv}).
The uniqueness of the extremal function follows from the uniqueness of the extremal configuration.

\begin{remark} We remark that
\[
 \re\frac{1-\varphi(0)}{1+\varphi(0)}\geq \frac{1-|\varphi(0)|}{1+|\varphi(0)|}.
 \]
\end{remark}

Inequality (\ref{ineqGuniv}) can be obtained with use of the theory of semigroups of holomorphic self-mappings of the unit disk.

Let $g$ be a generator of a one parameter semigroup $S=\{\varphi_t\}_{t\ge0}$ having a boundary fixed point $\xi=\varphi_t(\xi)\in\partial\mathbb{D}.$
If $g'(\xi)$ is finite, then it is a real number and $\varphi'_t(\xi)=e^{t g'(\xi)}$.
Since $\varphi_t$ univalent for each $t\geq 0$, one can write by (\ref{inegGuniv2})
\[
e^{t g'(\xi)}\ge\re\frac{\xi-\varphi_t(0)}{\xi+\varphi_t(0)}.
\]
We calculate
\begin{align*}
\frac{1-e^{tg'(\xi)}}{t}\le \frac{1-\re\frac{\xi-\varphi_t(0)}{\xi+\varphi_t(0)}}{t}=\frac{\re\frac{\xi+\varphi_t(0)-\xi+\varphi_t(0)}{\xi+\varphi_t(0)}}{t}=
\frac{\re\frac{2\varphi_t(0)}{\xi+\varphi_t(0)}}{t}=2\re\frac{\varphi_t(0)\bar{\xi}}{(1+\varphi_t(0)\bar{\xi})t}.
\end{align*}
Since $\lim_{t\rightarrow 0^{+}}(\varphi_t(z)-z)/t=g(z)$ and $\lim_{t\rightarrow0^{+}}\varphi_t(z)=z$, $z\in\mathbb{D}$, we obtain
\[
\lim_{t\rightarrow 0^{+}}\frac{1-e^{tg'(\xi)}}{t}\le2\lim_{t\rightarrow 0^{+}}\frac{\re \varphi_t(0)\bar{\xi}}{t}=2\re g(0)\bar{\xi}
\]
or, finally,
\begin{subequations}
\begin{equation}
\label{equ1}
-g'(\xi)\le 2\re g(0)\bar{\xi}.
\end{equation}
If $\xi=1$, then (\ref{equ1}) becomes
\begin{equation}
\label{equ2}
-g'(1)\le 2 \re g(0).
\end{equation}
\end{subequations}
Let now $\phi$ be any holomorphic function, $\phi(\mathbb{D})\subseteq \mathbb{D}$ with $\phi(\xi)=\xi\in\partial{\mathbb{D}}$.

Consider the function $g:\mathbb{D}\rightarrow\mathbb{C}$ defined by
\begin{equation}
\label{equg}
g(z):=(w-z)(1-z\bar{w})\frac{\xi-\phi(z)}{\xi+\phi(z)},
\end{equation}
where $w\in\partial\mathbb{D}\cup\mathbb{D}$ is chosen such that $w\neq \xi$ and $\phi\in\Hol(\mathbb D, \mathbb D)$.

It follows by the Berkson-Porta formula (see, for example, \cite{BerkPorta, Shoikhet}) that $g$ is a holomorphic generator and $g(\xi)=0$.
Therefore, one can use inequality (\ref{equ1}).

We calculate again
\begin{equation}
\label{long}
\begin{array}{rl}
\displaystyle{g'(\xi)=\lim_{z\rightarrow \xi}\frac{g(z)}{z-\xi}=\lim_{z\rightarrow \xi}(z-w)(1-z\bar{w})\frac{\phi(z)-\xi}{z-\xi}\cdot\frac{1}{\xi+\phi(z)}=}&\\
\xi|1-\bar{w}\xi|^2\phi'(\xi)\cdot\dfrac{1}{2\xi}=\dfrac{|1-\bar{w}\xi|^2}{2}\cdot \phi'(\xi).
\end{array}
\end{equation}
In addition, (\ref{equg}) implies that
\begin{equation}
\label{equg0}
g(0)=w\frac{\xi-\phi(0)}{\xi+\phi(0)}.
\end{equation}
We plug (\ref{long}) and (\ref{equg0}) into (\ref{equ1}) and get
\begin{equation}
-\frac{|1-\bar{w}\xi|^2}{2}\phi'(\xi)\le 2\re\left[\bar{\xi}w\,\,\frac{\xi-\phi(0)}{\xi+\phi(0)}\right].
\end{equation}
In particular, if $\xi=1$, one can choose $w=-1$ to get the inequality
\[
\phi'(1)\ge\re\frac{1-\phi(0)}{1+\phi(0)}
\]
which coincides with (\ref{ineqGuniv}).

\end{proof}

\subsection{Osserman type inequality} 

Combining our approach with the following result for univalent
self-mappings of the  disk established by Anderson and Vasil'ev
\cite{AV}, we derive some new estimate for holomorphic
self-mappings which are not necessarily univalent that, in
particular, improves  Osserman's result.

Let us define the Pick function
$$
p_{\beta}(z)=\frac{4\beta z}{\left(1-z+\sqrt{(1-z)^2+4\beta z}\right)^2}=\beta z+\dots,
$$
that maps the unit disk $\mathbb D$ onto $\mathbb D\setminus(-1,\,-\beta/(1+\sqrt{1-\beta})^2]$. Set the M\"obius transformation $$B_z(\zeta)=\frac{1-\bar{z}}{1-z}\frac{\zeta-z}{1-\zeta\bar{z}}.$$ 

\begin{theorem}\cite{AV} Let $\varphi \in\mathrm{Hol}(\mathbb{D},\mathbb{D})$ be a univalent function which is
conformal at the boundary point  $\xi =1$,
$\angle\lim\limits_{z\rightarrow 1}\varphi(z)=1$, and let
$\liminf\limits_{z\rightarrow
1}\displaystyle\frac{1-|\varphi(z)|}{1-|z|}=:\alpha$ exists and is finite.
Then for all $z\in\mathbb{D}$,
\begin{equation}\label{AV}
|\varphi'(z)|\geq \frac{1}{\alpha
^{2}}\frac{\left(1-|z|^{2}\right)^{3}}{|1-z|^{4}}\frac{|1-\varphi(z)|^{4}}{\left(1-|\varphi(z)|^{2}\right)^{3}}.
\end{equation}
With a fixed $z\in \mathbb D$ and $\varphi(z)=w$, the equality sign is attained only for the function $\varphi^*=B^{-1}_{w}\circ p_{\beta}\circ B_z$,
where
$$
\beta=\frac{1}{\alpha^2}\frac{(1-|z|^2)^2|1-\varphi^*(z)|^4}{|1-z|^4(1-|\varphi^*(z)|^2)^2}.
$$
\end{theorem}

\begin{theorem}
Let $\phi\in\mathrm{Hol}(\mathbb{D},\mathbb{D})$ and suppose that $\xi =1$ is its
boundary regular fixed point. Then
\begin{equation}\label{O1}
\phi'(1)\geq \frac{2}{\re
\frac{1-\phi(0)^{2}+\phi'(0)}{(1-\phi(0))^{2}}}.
\end{equation}
\end{theorem}

\begin{proof}
Let $S=\{\varphi_t\}_{t\geq 0}\subset\mathrm{Hol}(\mathbb{D},\mathbb{D})$ be a semigroup of conformal self-maps of $\mathbb D$ generated by
$g$, and
with a boundary regular fixed point $\xi =1$. Since all functions $\varphi_t$, $t\geq 0$, are univalent, they
satisfy inequality \eqref{AV}. We rewrite it in the form
\begin{equation}\label{O2}
\varphi'_{t}(1)^{2}\geq
\frac{1}{|\varphi_t'(z)|}\frac{\left(1-|z|^{2}\right)^{3}}{|1-z|^{4}}\frac{|1-\varphi_t(z)|^{4}}{\left(1-|\varphi_t(z)|^{2}\right)^{3}}.
\end{equation}
Obviously, the map $\varphi_t\big|_{t=0^+}=id$ gives the equality in \eqref{O2}. Then the same inequality holds for the $t$-derivatives of the both sides of \eqref{O2} at
the point $0^+$.
The derivative of the left-hand side at $t=0^{+}$ becomes
\begin{equation}
\frac{d}{dt}\varphi'_{t}(1)^{2}\bigg|_{t=0^+}=\frac{d}{dt}e^{2tg'(1)}\bigg|_{t=0^+}=2g'(1).
\end{equation}
In the right-hand side, the derivative is equal to the limit
\begin{equation}
\begin{split}
R:&=\lim\limits _{t\rightarrow
0^{+}}\frac{\displaystyle\frac{1}{|\varphi'_{t}(z)|}\displaystyle\frac{\left(1-|z|^{2}\right)^{3}}{|1-z|^{4}}\displaystyle\frac{|1-\varphi_t(z)|^{4}}{\left(1-|\varphi_t(z)|^{2}\right)^{3}}-1}{t}\\
&=\lim\limits _{t\rightarrow
0^{+}}\frac{\left(1-|z|^{2}\right)^{3}|1-\varphi_t(z)|^{4}-|\varphi'_{t}(z)|
|1-z|^{4}\left(1-|\varphi_t(z)|^{2}\right)^{3}}{|\varphi'_{t}(z)|
|1-z|^{4}\left(1-|\varphi_t(z)|^{2}\right)^{3} t}\\ &=\lim\limits
_{t\rightarrow
0^{+}}\frac{\left(1-|z|^{2}\right)^{3}|1-\varphi_t(z)|^{4}-|\varphi'_{t}(z)|
|1-z|^{4}\left(1-|\varphi_t(z)|^{2}\right)^{3}}{|1-z|^{4}\left(1-|z|^{2}\right)^{3}
t}\\&=\lim\limits _{t\rightarrow
0^{+}}\frac{\left(1-|z|^{2}\right)^{6}|1-\varphi_t(z)|^{8}-|\varphi'_{t}(z)|^{2}
|1-z|^{8}\left(1-|\varphi_t(z)|^{2}\right)^{6}}{2|1-z|^{8}\left(1-|z|^{2}\right)^{6}
t}
\end{split}
\notag
\end{equation}
Since for all $z\in\mathbb{D}$, $\lim\limits _{t\rightarrow
0^{+}}\displaystyle\frac{\varphi_t(z)-z}{t}=g(z)$ and $\lim\limits
_{t\rightarrow 0^{+}}\displaystyle\frac{\varphi'_{t}(z)-1}{t}=g'(z)$,
the first term in the numerator is
\begin{equation}
\begin{split}
&\left(1-|z|^{2}\right)^{6}|1-\varphi_t(z)|^{8}=\left(1-|z|^{2}\right)^{6}|(1-z)+(z-\varphi_t(z))|^{8}\\
&=\left(1-|z|^{2}\right)^{6}\left(|1-z|^{2}+2\re
(1-\overline{z})(z-\varphi_t(z))+o(t) \right)^{4}\\
&=\left(1-|z|^{2}\right)^{6}\left(|1-z|^{4}+4|1-z|^{4}\re
(1-\overline{z})(z-\varphi_t(z))+o(t) \right)^{2}\\
&=\left(1-|z|^{2}\right)^{6}\left(|1-z|^{8}+8|1-z|^{8}\re
(1-\overline{z})(z-\varphi_t(z))+o(t) \right),\\
\end{split}
\notag
\end{equation}
and the second term is
\begin{equation}
\begin{split}
&|\varphi'_{t}(z)|^{2}|1-z|^{8}\left(1-|\varphi_t(z)|^{2}\right)^{6}\\
&=|1-z|^{8}\left(1+2\re(\varphi'_{t}(z)-1)+o(t)\right)\left(1-|z|^{2}-2\re\overline{z}(\varphi_t(z)-z)+o(t)\right)^{6}\\
&=|1-z|^{8}\left(1+2\re
\left(\varphi'_{t}(z)-1\right)\right)\left(\left(1-|z|^{2}\right)^{6}-12\left(1-|z|^{2}\right)^{5}\re
\overline{z}(\varphi_t(z)-z)
\right)+o(t)\\&=|1-z|^{8}\left(\left(1-|z|^{2}\right)^{6}-12\left(1-|z|^{2}\right)^{5}\re
\overline{z}(\varphi_t(z)-z)+2\left(1-|z|^{2}\right)^{6}\re
\left(\varphi'_{t}(z)-1\right) \right)+o(t).
\end{split}
\notag
\end{equation}
Hence,
\begin{equation}
\begin{split}
&R=\lim\limits _{t\rightarrow
0^{+}}\frac{8\left(1-|z|^{2}\right)\re
(1-\overline{z})(z-\varphi_t(z))+12\re
\overline{z}(\varphi_t(z)-z)-2\left(1-|z|^{2}\right)\re (\varphi'_{t}(z)-1)
}{2\left(1-|z|^{2}t\right)}\\&=\frac{-8\left(1-|z|^{2}\right)\re
(1-\overline{z})g(z)+12\re
\overline{z}g(z)-2\left(1-|z|^{2}\right)\re
g'(z)}{2\left(1-|z|^{2}\right)}.
\end{split}
\notag
\end{equation}
Consequently, for all $z\in\mathbb{D}$,
\begin{equation}
2g'(1)\geq \frac{-8\left(1-|z|^{2}\right)\re
(1-\overline{z})g(z)+12\re
\overline{z}g(z)-2\left(1-|z|^{2}\right)\re
g'(z)}{2\left(1-|z|^{2}\right)}. \notag
\end{equation}
In particular, for $z=0$,
\begin{equation}\label{O3}
2g'(1)\geq -\re \left(g'(0)+4g(0)\right).
\end{equation}
Suppose $\phi\in\mathrm{Hol}(\mathbb{D},\mathbb{D})$ with the regular fixed point
$\xi =1$. Then the function
$$g(z)=(1-z)^{2}\frac{1+\phi(z)}{1-\phi(z)}$$ is a generator of a semigroup with the
boundary regular fixed point $\xi =1$, and so it satisfies
inequality \eqref{O3}. Simple calculations show that
\begin{equation}
g'(1)=-\frac{2}{\phi'(1)}, \quad
g'(0)=\frac{2\phi'(0)}{(1-\phi(0))^{2}}-2\frac{1+\phi(0)}{1-\phi(0)}, \quad
g(0)=\frac{1+\phi(0)}{1-\phi(0)}. \notag
\end{equation}
Substituting these expressions in \eqref{O3}, we have \eqref{O1}.
\end{proof}

An analogous estimate for angular derivatives including the values
of $\phi(0)$, $\phi'(0)$ and $\phi'(1)$, was established by Osserman in
\cite{Osserman}. For the case of the boundary regular fixed point
$\xi =1$, we write his estimate in the form which is
convenient to compare with our result:
\begin{equation}
\phi'(\xi )\geq
\frac{2}{\displaystyle\frac{1-|\phi(0)|^{2}+|\phi'(0)|}{(1-|\phi(0)|)^{2}}}=:L.
\end{equation}
It is easy to see that our inequality improves this estimate. Moreover, Osserman's
estimate is the same for all boundary fixed points. Hence, if the
Denjoy--Wolf point of $\phi$ is situated on the boundary of the unit
disk, then $L\leq 1$ and, consequently, for all other (repelling)
boundary fixed points, it does not give additional information. Thus, in a simple
example of the hyperbolic automorphism
$\phi(z)=\displaystyle\frac{2z-1}{2-z}$ of $\mathbb{D}$ with fixed points
at $\pm 1$, $\phi'(-1)=\displaystyle\frac{1}{3}$ and $\phi'(1)=3$,
Osserman's estimate gives $\phi'(1)\geq \displaystyle\frac{1}{3}$
whereas \eqref{O1} implies that $\phi'(1)\geq 3$.

Further, we present some estimates for holomorphic self-mappings
of the unit disk whose images do not contain the origin, i.e.,
$0\notin \phi(\mathbb{D})$. We also derive them from some inequalities for
generators.
\begin{theorem}
Let $\phi\in\mathrm{Hol}(\mathbb{D},\mathbb{D})$. Suppose that $\phi(\mathbb{D} )$ does not
separate the origin and the boundary $\partial\mathbb{D}$ and
$0\notin \phi(\mathbb{D})$. If $\xi =1$ is a boundary regular fixed
point of $\phi$ then
\begin{equation}\label{origin1}
\phi'(1)\geq -\frac{\ln |\phi(0)|}{2}.
\end{equation}
\end{theorem}
\begin{proof}
For a function $\phi$ which satisfies our assumptions, define
the generator
\begin{equation}\label{origin2}
g(z)=-z\frac{1-\phi(z)^{\frac{1}{n}}}{1+\phi(z)^{\frac{1}{n}}}, \quad
z\in\mathbb{D} ,
\end{equation}
where $w^{\frac{1}{n}}$ means the analytic branch of the root fixing $1$. The generator $g$ has a
regular null point at $1$, and therefore, satisfies inequality \eqref{O3}.
Substituting
\begin{equation}
g(0)=0, \quad
g'(0)=-\frac{1-\phi(0)^{\frac{1}{n}}}{1+\phi(0)^{\frac{1}{n}}},\quad
\mbox{and} \quad g'(1)=\frac{1}{n}\frac{\phi'(1)}{2}, \notag
\end{equation}
in \eqref{O3}, we have $$\phi'(1)\geq n\,\re
\frac{1-\phi(0)^{\frac{1}{n}}}{1+\phi(0)^{\frac{1}{n}}},$$ and in the
limit case, as $n\rightarrow\infty$, we have \eqref{origin1}.
\end{proof}

Notice, that if $\phi(0)\in (0,1)$, then one can show that inequality
\eqref{origin1} improves the known estimate $$\phi'(1)\geq
\frac{1}{\re \displaystyle\frac{1+\phi(0)}{1-\phi(0)}},$$ which holds for
holomorphic self-mappings of $\mathbb{D}$ without an additional
restriction on the image $\phi(\mathbb{D} )$ (see
\cite{CowenPommerenke}).

A similar method with generator \eqref{origin2} applied to the
Harnack inequality for a holomorphic function $p$ on $\mathbb{D}$ with
 positive real part (see, for example, \cite{Shoikhet})
\begin{equation}\label{Har}
\frac{1-|z|}{1+|z|}\re p(0)\leq \re p(z)\leq
\frac{1+|z|}{1-|z|}\re p(0),
\end{equation}
gives again some estimate for holomorphic self-maps of the disk
such that $0\notin \phi(\mathbb{D})$.
\begin{theorem}
Let $\phi\in\mathrm{Hol}(\mathbb{D},\mathbb{D})$. Suppose that $\phi(\mathbb{D} )$ does not
separate the origin and the boundary $\partial\mathbb{D}$ and $0\notin
\phi(\mathbb{D})$. Then for all $z\in \mathbb{D}$,
\begin{equation}
|\phi(0)|^{\frac{1+|z|}{1-|z|}}\leq |\phi(z)|\leq
|\phi(0)|^{\frac{1-|z|}{1+|z|}}.
\end{equation}
\end{theorem}
\begin{proof}
Suppose that a function $\phi$ satisfies our assumptions. Then for
$$p(z)=\frac{1-\phi(z)^{\frac{1}{n}}}{1+\phi(z)^{\frac{1}{n}}}, \quad
n\in\mathbb{N}, \quad z\in\mathbb{D} ,$$ inequality \eqref{Har} holds,
and so $$\frac{1-|z|}{1+|z|}n\, \re
\frac{1-\phi(0)^{\frac{1}{n}}}{1+\phi(0)^{\frac{1}{n}}}\leq n\, \re
\frac{1-\phi(z)^{\frac{1}{n}}}{1+\phi(z)^{\frac{1}{n}}}\leq
\frac{1+|z|}{1-|z|} n\, \re
\frac{1-\phi(0)^{\frac{1}{n}}}{1+\phi(0)^{\frac{1}{n}}}.$$ Since for
all $z\in\mathbb{D}$, $$\lim\limits _{n\rightarrow \infty}
n\frac{1-\phi(z)^{\frac{1}{n}}}{1+\phi(z)^{\frac{1}{n}}}=-\frac{\log
\phi(z) }{2},$$ we have $$-\frac{1-|z|}{1+|z|}\re \log \phi(0)\leq -\re
\log \phi(z)\leq -\frac{1+|z|}{1-|z|}\re \log \phi(0), $$ or which is
the same,
\begin{equation}
|\phi(0)|^{\frac{1+|z|}{1-|z|}}\leq |\phi(z)|\leq
|\phi(0)|^{\frac{1-|z|}{1+|z|}}. \notag
\end{equation}
\end{proof}

Observe, that the right-hand side of this inequality improves the
known Lindel\"of's estimate \cite{LE-09} (which holds for all
holomorphic self-mappings of $\mathbb{D}$ without an additional
assumption on the image $\phi(\mathbb{D})$): $$|\phi(z)|\leq
\frac{|z|+|\phi(0)|}{1+|z||\phi(0)|}.$$ Indeed, for each $a\in (0,1)$,
the real-valued functions $$\varphi (x)=a^{x} \quad \mbox{and}
\quad \psi (x)=\frac{(1-x)+a(1+x)}{(1+x)+a(1-x)}, \quad x\in
(0,1),$$ are strictly decreasing and coincide at the points $x=0$
and $x=1$, whereas $\varphi \left(\frac{1}{2}\right)<\psi
\left(\frac{1}{2}\right)$. So, $\varphi (x)<\psi (x)$ for all
$x\in (0,1)$. Setting $x=\frac{1-|z|}{1+|z|}$ and $a=|\phi(0)|$, we
have
$$|\phi(0)|^{\frac{1-|z|}{1+|z|}}<\frac{|z|+|\phi(0)|}{1+|z||\phi(0)|}$$
for all $0\neq z\in\mathbb{D} $.

\section{Two fixed boundary points}
\label{sec:two}
In this section we obtain estimates for angular derivatives of  a univalent self-map of the unit disk with two fixed boundary points.

Let $\xi_1$ and $\xi_2$ be two different points of the unit circle $\mathbb T$. Let $\varphi\colon \mathbb D\to\mathbb D$ be univalent and $\varphi(\xi_j)=\xi_j$, $j=1,2$. We are interested in the lower estimate of the product $\varphi'(\xi_1)\varphi'(\xi_2)$ for functions $\varphi$ with a fixed value of $\varphi(0)\in\mathbb D$.
Again without loss of generality, applying rotation, we assume that  $\xi_1=e^{-i\theta}$,  $\xi_2=e^{i\theta}$, and $\theta\in (0,\frac{\pi}{2}]$.

In order to formulate the main result we need several preparatory definitions and notations.
For every $a\in\mathbb D$ let us define a real-valued continuous function $\Phi(a)$ by
\[
\Phi(a)=
\begin{cases}
\frac{1-|a|^2}{2(\cos\theta-\re a)}-\sqrt{1+\left(\frac{1-|a|^2}{2(\cos\theta-\re a)}\right)^2-\frac{1-|a|^2}{\cos\theta-\re a}\cos\theta}, & \text{if $\re a<\cos\theta$};\\
\cos\theta, & \text{if $\re a=\cos\theta$};\\
\frac{1-|a|^2}{2(\cos\theta-\re a)}+\sqrt{1+\left(\frac{1-|a|^2}{2(\cos\theta-\re a)}\right)^2-\frac{1-|a|^2}{\cos\theta-\re a}\cos\theta}, & \text{if $\re a>\cos\theta$}.
\end{cases}
\]
Geometrically, this function defines the intersection of the arc of a circle containing the points $a$ and $e^{\pm i\theta}$ with the interval $(-1,1)$.
\begin{figure}[h!]
  \centering
    \includegraphics[width=0.5\textwidth]{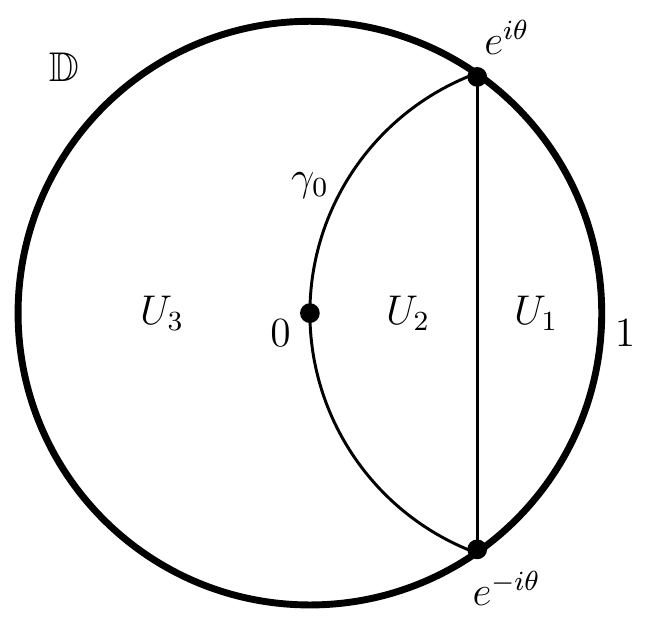}
  \caption{Domains $U_1$, $U_2$, and $U_3$ for $\varphi(0)$}
\end{figure}

Denote by $\gamma_0$ the arc of the circle $\{(x,y)\colon (x^2+y^2)\cos\theta=x\}$ inside the unit disk $\mathbb D$, which contains the origin and the points $e^{\pm i\theta}$ in its closure. Observe that if $a\in \gamma_0$, then $\Phi(a)=0$, and if $a$ is real, then $\Phi(a)=a$. Let us define three domains $U_1$, $U_2$, and $U_3$ of range of  $\varphi(0)$ as in Fig.1.  By $U_1$ we denote the segment between the unit circle to the right of the interval
connecting the points $e^{\pm i\theta}$. By $U_2$ we denote the segment between the  circle $\gamma_0$ to the left of the interval
connecting the points $e^{\pm i\theta}$. Finally, by $U_3$ we denote the domain between the unit circle to the left from $\gamma_0$.

Now let us define the future extremal functions. Let $\zeta(z)=z+\frac{1}{z}$ denote the Joukowski map. For a fixed value of $x\in [0,1)$ we define the Pick function $p^+_x(z)$ as a superposition of three functions $p^+_x(z)=\zeta^{-1}\circ u^+\circ \zeta(z)$, where $u^+(\zeta)$ is a M\"obius map given by
\[
u^+(\zeta)=\frac{(1+x^2)\zeta-4x\cos\theta}{x\zeta+(1-x)^2-2x\cos\theta}.
\]
Observe that $p^+_x(e^{\pm i\theta})=e^{\pm i\theta}$ and the function $p^+_x(z)$ maps the disk $\mathbb D$ onto the domain $\mathbb D\setminus (-1,r^+]$, where
\[
r^+=\frac{2\sqrt{2x(1+x^2)\cos\theta}-x^2-2x\cos\theta-1}{x^2-2x\cos\theta+1}.
\]
The angular derivative is
\[
(p_x^+)'(e^{\pm i\theta})=\frac{1-2x\cos\theta+x^2}{(1-x)^2}.
\]

Analogously, for a fixed value of $x\in (-1,0]$ we define the Pick function $p^-_x(z)$ as a superposition of three functions $p^-_x(z)=\zeta^{-1}\circ u^-\circ \zeta(z)$, where $u^-(\zeta)$ is a M\"obius map given by
\[
u^-(\zeta)=\frac{-(1+x^2)\zeta+4x\cos\theta}{x\zeta-(1-x)^2-2x\cos\theta}.
\]
Observe that again $p^-_x(e^{\pm i\theta})=e^{\pm i\theta}$ and the function $p^-_x(z)$ maps the disk $\mathbb D$ onto the domain $\mathbb D\setminus [r^-_,1)$, where
\[
r^-=\frac{x^2-2x\cos\theta+1-2\sqrt{2x(1+x^2)\cos\theta}}{x^2-2x\cos\theta+1}.
\]
The angular derivative is
\[
(p_x^-)'(e^{\pm i\theta})=\frac{1-2x\cos\theta+x^2}{(1+x)^2}.
\]
We notice that if $x=0$, then $p^+_0(z)\equiv p^-_0(z)\equiv z$.

Finally, let us define the map $\mob(z,t)$, $t\in(-\theta,\theta)$ given as a solution to the equation
\[
\frac{z-e^{-i\theta}}{z-e^{i\theta}}\frac{1-e^{i\theta}}{1-e^{-i\theta}}= \frac{w-e^{-i\theta}}{w-e^{i\theta}}\frac{e^{it}-e^{i\theta}}{e^{it}-e^{-i\theta}},
\]
and a special value of $t=t_0=t_0(a)$ given a unique solution in the interval $t\in(-\theta,\theta)$  to the equation
\[
\frac{a-e^{-i\theta}}{a-e^{i\theta}}\frac{1-e^{i\theta}}{1-e^{-i\theta}}= \frac{\Phi(a)-e^{-i\theta}}{\Phi(a)-e^{i\theta}}\frac{e^{it}-e^{i\theta}}{e^{it}-e^{-i\theta}}.
\]
By writing $\mob^{-1}(z,t)$ we mean the inverse with respect to $z$.

Let us remind that if $\varphi \in \mathrm{Hol}(\mathbb{D},\mathbb{D})$ has no interior fixed points, then there exists a boundary Denjoy-Wolff point $\xi_1$ in which $\varphi'(\xi_1)\in (0,1]$. If there is another boundary fixed point $\xi_2$, then $\varphi'(\xi_2) \geq 1$.  It is a simple consequence from Julia-Carath\'eodory-Wolff results that
\begin{equation}\label{12}
\varphi'(\xi_1)\varphi'(\xi_2)\geq 1,
\end{equation}
see also \cite[Theorem 3.1]{CowenPommerenke}. The following theorem is a refinement of \eqref{12}.

\begin{theorem}\label{MainTheorem}
If $\varphi \in \mathrm{Hol}(\mathbb{D},\mathbb{D})$ and univalent, $\varphi(e^{\pm i\theta})=e^{\pm i\theta}$, $\theta\in (0,\frac{\pi}{2})$ then the following sharp estimates hold:
\begin{itemize}
\item[(a)] If $\varphi(0)\in U_1\cup U_2\cup [e^{-i\theta},e^{i\theta}]$ (lies to the right of $\gamma_0$), then
\[
\sqrt{\varphi'(e^{i\theta})\varphi'(e^{-i\theta})}\geq \frac{1-2\Phi(\varphi(0))\cos\theta+\Phi^2(\varphi(0))}{(1- \Phi(\varphi(0))^2};
\]
\item[(b)] If $\varphi(0)\in U_3$ (lies to the left  of $\gamma_0$), then
\[
\sqrt{\varphi'(e^{i\theta})\varphi'(e^{-i\theta})}\geq \frac{1-2\Phi(\varphi(0))\cos\theta+\Phi^2(\varphi(0))}{(1+ \Phi(\varphi(0))^2}.
\]
\item[(c)] If  $\varphi(0)\in\gamma_0$, then $\varphi'(e^{i\theta})\varphi'(e^{-i\theta})\geq 1$, or equivalently, the inequality \eqref{12} can not be refined.
\end{itemize}
In all three cases the estimates are sharp and the extremal value is given by a unique extremal function:
\begin{itemize}
\item[(a)] $\mob^{-1}\left(p^+_{\Phi(\varphi(0))}(z),t_0(\varphi(0))\right)$;
\item[(b)] $\mob^{-1}\left(p^-_{\Phi(\varphi(0))}(z),t_0(\varphi(0))\right)$;
\vspace{5pt}
\item[(c)] $\mob^{-1}\left(z,t_0(\varphi(0))\right)$.
\end{itemize}
\end{theorem}

\begin{corollary}
If $\varphi \in \mathrm{Hol}(\mathbb{D},\mathbb{D})$ and univalent, $\varphi(e^{\pm i\theta})=e^{\pm i\theta}$, $\theta\in (0,\frac{\pi}{2})$ then the following sharp estimates hold:
\begin{itemize}
\item If $\varphi(0)\in U_1$, then
\[
\sqrt{\varphi'(e^{i\theta})\varphi'(e^{-i\theta})}\geq \frac{1-2\re\varphi(0)\cos\theta+(\re\varphi(0))^2}{(1- \re\varphi(0))^2};
\]
\item If $\varphi(0)\in U_3$, then
\[
\sqrt{\varphi'(e^{i\theta})\varphi'(e^{-i\theta})}\geq  \frac{1-2\re\varphi(0)\cos\theta+(\re\varphi(0))^2}{(1+ \re\varphi(0))^2}.
\]
\end{itemize}
Observe that if $\theta=\frac{\pi}{2}$, then the domain $U_2$ degenerates. The equality sign is attained for real values of $\varphi(0)$ and for $p^+_{\varphi(0)}(z)$ and $p^-_{\varphi(0)}(z)$ respectively.
\end{corollary}

\begin{proof}

Let us assume first that $\theta\in(0,\frac{\pi}{2})$, and let $\varphi\in \Hol(\mathbb D,\mathbb D)$ fixes two points $\xi_1=e^{-i\theta}$ and  $\xi_2=e^{i\theta}$, and let finally
$\varphi(0)\in U_1\cup U_2\cup [e^{-i\theta},e^{i\theta}]$, see~Fig~1.

Let us consider the family $\mathcal F^{1}_{a_0}$ of digons $D$ in $\Omega=\mathbb D\setminus \{a_0\}$, $0<a_0<1$ with two vertices $e^{\pm i\theta}$  and with the equal angles $\pi$ at these vertices,  such that
any arc connecting $e^{i\theta}$ and $e^{-i\theta}$ in $D$ is homotopic in $\Omega$ to the arc $\{e^{it}\colon t\in(-\theta,\theta)\}$.\

Then
\[
\min\limits_{D\in \mathcal F^{1}_{a_0}} m(D,e^{i\theta},e^{-i\theta})= m(D^{1}_{a_0},e^{i\theta},e^{-i\theta}),
\]
where $D_{a_0}^{1}$ is a strip domain in the trajectory structure of the quadratic differential
\[
Q(w)dw^2=\frac{A^2(w+1)^2}{(w-e^{i\theta})^2(w-e^{-i\theta})^2(w-a_0)(1-{a_0}w)}dw^2,\quad w\in \Omega,
\]
for some real value of $A$.

Using elementary conformal maps and known reduced moduli of
digons, see \cite[page 33]{VasBook}, we calculate the modulus of $D_{a_0}^1$ as
\[
m(D^{1}_{a_0},e^{i\theta},e^{-i\theta})=\frac{2}{\pi}\log\frac{4(1-\cos\theta)}{\sin\theta}\frac{1-2a_0\cos\theta+a_0^2}{(1-a_0)^2}.
\]

The digon $D^1_0$ is mapped by the function $\varphi$, satisfying the above properties, onto the digon $\varphi(D_0^1)$ with the same vertices and angles at them, and
$\varphi(0)\in U_1$. Unfortunately, it is not possible to apply symmetrization to this digon because the resulting object will not be a digon with the same vertices. Therefore, we
apply another procedure.

Let us consider the following M{\"o}bius map $w=\mob(z,t)$
\[
\frac{z-e^{-i\theta}}{z-e^{i\theta}}\frac{1-e^{i\theta}}{1-e^{-i\theta}}= \frac{w-e^{-i\theta}}{w-e^{i\theta}}\frac{e^{it}-e^{i\theta}}{e^{it}-e^{-i\theta}},
\]
which makes the correspondence $1\to e^{it}$, $e^{\pm i\theta}\to e^{\pm i\theta}$. Fix a point $a\in\mathbb D$ and consider the curve $\gamma\colon (-\theta,\theta)\to \mathbb D$
passing through $a=\gamma(0)$ and defined by the equation
\[
\frac{a-e^{-i\theta}}{a-e^{i\theta}}\frac{1-e^{i\theta}}{1-e^{-i\theta}}= \frac{\gamma(t)-e^{-i\theta}}{\gamma(t)-e^{i\theta}}\frac{e^{it}-e^{i\theta}}{e^{it}-e^{-i\theta}}.
\]
We have $\lim_{t\to -\theta+0}\gamma(t)=e^{-i\theta}$, and $\lim_{t\to \theta-0}\gamma(t)=e^{i\theta}$
Observe that $\gamma$ is an arc of a circle centered on the point
\[
\left(\frac{1-|a|^2}{2(\cos\theta-\re a)},0\right),
\]
and of radius
\[
\sqrt{1+\left(\frac{1-|a|^2}{2(\cos\theta-\re a)}\right)^2-\frac{1-|a|^2}{\cos\theta-\re a}\cos\theta}.
\]
If $\re a=\cos\theta$, then the arc $\gamma$ becomes the interval $[e^{i\theta},e^{-i\theta}]$.
The arc $\gamma$ intersects the real axis inside the unit disk at the point
$
\gamma(t_0)=\Phi(a),
$
and we denote $t_0=\gamma^{-1}(\Phi(a))$. We remark that if $a$ is real, then $\Phi(a)=a$. If $a\in U_1\cup U_2\cup [e^{-i\theta},e^{i\theta}]$, then $\Phi(a)\in(0,1)$. If $a\in U_3$, then $\Phi(a)\in(-1,0)$. If $a\in \gamma_0$, then $\Phi(a)=0$.

The angular derivatives are
\[
w'_z(e^{-i\theta},t_0)=\frac{\sin\frac{\theta-t_0}{2}}{\sin\frac{\theta+t_0}{2}}, \quad w'_z(e^{i\theta},t_0)=\frac{\sin\frac{\theta+t_0}{2}}{\sin\frac{\theta-t_0}{2}}.
\]

Let $t_0$ be defined as $t_0=\gamma^{-1}(\Phi(\varphi(0)))$. Observe that $\Phi(\varphi(0))\in (0,1)$. The M\"obius transformation $\mob(z,t_0)$ maps the digon $\varphi(D_0^1)$
onto the digon $\mob(\varphi(D_0^1),t_0)$, which is admissible in the problem for the family of digons $\mathcal F_{\Phi(\varphi(0))}^1$.  Due to admissibility we can write
\[
m(D^{1}_{0},e^{i\theta},e^{-i\theta})+\frac{1}{\pi}\log \varphi'(e^{i\theta})\varphi'(e^{-i\theta})+\frac{1}{\pi}\log w'_z(e^{i\theta},t_0)w'_z(e^{-i\theta},t_0)\geq m(D^{1}_{\Phi(\varphi(0))},e^{i\theta},e^{-i\theta}).
\]
Then we arrive at following inequality (a) of Theorem 1
\begin{equation}\label{main1}
\sqrt{\varphi'(e^{i\theta})\varphi'(e^{-i\theta})}\geq \frac{1-2\Phi(\varphi(0))\cos\theta+\Phi^2(\varphi(0))}{(1- \Phi(\varphi(0))^2}
\end{equation}

Now let $\varphi(0)\in U_3$, see~Fig~1. Let us consider the family $\mathcal F^{2}_{a_0}$ of digons $D$ in $\Omega=\mathbb D\setminus \{a_0\}$, $-1<a_0<0$ with two vertices $e^{\pm i\theta}$  and with the equal angles $\pi$ at these vertices,  such that
any arc connecting $e^{i\theta}$ and $e^{-i\theta}$ in $D$ is homotopic in $\Omega$ to the arc $\{e^{it}\colon t\in(\theta,2\pi-\theta)\}$.

Then
\[
\min\limits_{D\in \mathcal F^{2}_{a_0}} m(D,e^{i\theta},e^{-i\theta})= m(D^{2}_{a_0},e^{i\theta},e^{-i\theta}),
\]
where $D_{a_0}^{2}$ is a strip domain in the trajectory structure of the quadratic differential
\[
Q(w)dw^2=\frac{A^2(w-1)^2}{(w-e^{i\theta})^2(w-e^{-i\theta})^2(w-a_0)(1-{a_0}w)}dw^2,\quad w\in \Omega,
\]
for some real value of $A$.

Using elementary conformal maps and known reduced moduli of
digons, see \cite[page 33]{VasBook}, we calculate the modulus of $D_{a_0}^2$ as
\[
m(D^{2}_{a_0},e^{i\theta},e^{-i\theta})=\frac{2}{\pi}\log\frac{4(1-\cos\theta)}{\sin\theta}\frac{1-2a_0\cos\theta+a_0^2}{(1+a_0)^2}.
\]

Let $t_0$ be again defined as $t_0=\gamma^{-1}(\Phi(\varphi(0)))$. Observe that $\Phi(\varphi(0))\in (-1,0)$. The M\"obius transformation $\mob(z,t_0)$ maps the digon $\varphi(D_0^2)$
onto the digon $\mob(\varphi(D_0^2),t_0)$, which is admissible in the problem for the family of digons $\mathcal F_{\Phi(\varphi(0))}^2$.  Due to admissibility we can write
\[
m(D^{2}_{0},e^{i\theta},e^{-i\theta})+\frac{1}{\pi}\log \varphi'(e^{i\theta})\varphi'(e^{-i\theta})+\frac{1}{\pi}\log w'_z(e^{i\theta},t_0)w'_z(e^{-i\theta},t_0)\geq m(D^{2}_{\Phi(\varphi(0))},e^{i\theta},e^{-i\theta}).
\]
Then we arrive at following inequality (b) of Theorem 1
\begin{equation}\label{main2}
\sqrt{\varphi'(e^{i\theta})\varphi'(e^{-i\theta})}\geq \frac{1-2\Phi(\varphi(0))\cos\theta+\Phi^2(\varphi(0))}{(1+ \Phi(\varphi(0))^2}
\end{equation}

If $\varphi(0)\in \gamma_0$, then $\Phi(\varphi(0))=0$, and the sharp inequality $$\varphi(e^{i\theta})\varphi(e^{-i\theta})\geq 1$$ holds
with the equality for the M{\"o}bius map $\mob(z,t_0)$, where $t_0$ is defined by the formula
\[
e^{it_0}=\frac{1-\varphi(0)}{1+(1-2\cos\theta)\varphi(0)}.
\]
Observe that $\varphi(0)\in\gamma_0$, hence the absolute value of the right-hand side of the above equation is one.

In order to prove the Corollary we consider
the domains $U_1$, $U_2$, and $U_3$, as well as the arc $\gamma_0$, which are defined in Introduction.
If $\varphi(0)\in U_1$, then $\Phi(\varphi(0))>\re \varphi(0)>\cos\theta$ and the inequality \eqref{main1} may be strengthened as
\[
\sqrt{\varphi'(e^{i\theta})\varphi'(e^{-i\theta})}\geq \frac{1-2\re\varphi(0)\cos\theta+(\re\varphi(0))^2}{(1- \re\varphi(0))^2}.
\]
If $\varphi(0)\in U_3$, then $\Phi(\varphi(0))<\re \varphi(0)<0$ and
the inequality \eqref{main2} may be strengthened as
\[
\sqrt{\varphi'(e^{i\theta})\varphi'(e^{-i\theta})}\geq  \frac{1-2\re\varphi(0)\cos\theta+(\re\varphi(0))^2}{(1+ \re\varphi(0))^2}.
\]
In both inequalities the equality sign is attained only if $\varphi(0)$ is real.

In order to construct the extremal functions we observe that the function $p^{\pm}_x$ satisfies the equation
\[
\frac{(w\pm 1)^2}{(w-e^{i\theta})^2(w-e^{-i\theta})^2(w-x)(1-xw)}dw^2=\frac{(z\pm 1)^2}{(z-e^{i\theta})^2(z-e^{-i\theta})^2z}dz^2
\]
in $D_0^1$ or $D_0^2$ respectively, i.e., map the extremal configuration for $\mathcal F^{1,2}_0$ onto the extremal configuration for $\mathcal F^{1,2}_x$.
Further application of corresponding M\"obius transforms for $x=\Phi(\varphi(0))$ finishes the proof.
\end{proof}

\end{document}